\documentclass[preprint,authoryear,12pt]{elsarticle}

\usepackage{amssymb}

\usepackage{amsfonts}
\usepackage{amsmath}
\usepackage{float}
\usepackage{graphicx}
\usepackage{amsmath,amsthm,amssymb,amscd}
\usepackage{amsfonts}
\usepackage{amsmath,amssymb,amsthm,enumerate,epsfig,graphicx}
\usepackage{ifthen,latexsym,syntonly}
\usepackage{natbib}
\usepackage{verbatim}
\usepackage[hypertex]{hyperref}
\usepackage{enumerate}
\usepackage{paralist}
\usepackage{array}

\newtheorem{theorem}{Theorem}

\newtheorem{corollary}[theorem]{Corollary}

\newtheorem{lemma}[theorem]{Lemma}

\newtheorem{proposition}[theorem]{Proposition}

\newtheorem{asmptn}[theorem]{Assumption}

\def\fE {{\mathbf E}}

\def\fP {{\mathbf P}}

\def\to{\rightarrow}

\def\q{\quad}

\def\phi{\varphi}

\DeclareMathOperator{\var}{Var}

\rmfamily

\journal{Stochastic Processes and their Applications}

\begin{document}

\begin{frontmatter}



\title{Subcritical branching processes in random environment without Cramer
condition}


\author[VV]{Vladimir Vatutin\fnref{label1}}\ead{vatutin@mi.ras.ru}
\author[XZ]{Xinghua Zheng\corref{cor1}\fnref{label2}}\ead{xhzheng@ust.hk}

\address[VV]{Department of Discrete Mathematics, Steklov Mathematical Institute, 8 Gubkin
Street, 119\,991 Moscow, Russia\fnref{label1}}
\fntext[label1]{Research partially
supported by the Russian Foundation for Basic Research, grant 11-01-00139.}
\address[XZ]{Department of ISOM,
Hong Kong University of Science and Technology,
Clear Water Bay, Kowloon, Hong Kong SAR\fnref{label2}}
\fntext[label2]{Research partially supported
by GRF 606010 of the HKSAR.}

\cortext[cor1]{Tel: +852 2358 7750; Fax: +852 2358 2421}

\begin{abstract}
A subcritical branching process in random environment (BPRE) is considered whose associated random walk does not satisfy the Cramer condition. The asymptotics for the survival probability of the process is investigated, and a Yaglom type conditional limit theorem is proved for the number of particles up to moment $n$ given survival to this moment. Contrary to other types of subcritical BPRE, the limiting distribution is not discrete. We also show that the process survives for a long time owing to a single big jump of the associate random walk accompanied by a population explosion at the beginning of the process.

\end{abstract}


\end{frontmatter}



\section{Introduction and Main Results}

In this paper we consider the asymptotic behavior of a type of subcritical
branching processes in random environment. More specifically, the random
environment is given by a sequence of independent and identically
distributed (i.i.d.) probability distributions on nonnegative integers,
denoted by $\pi =\left\{ \mathbf{\pi }_{n},\;n\geq 0\right\}$ where
\begin{equation*}
\mathbf{\pi }_{n}=\left\{ \pi _{n}^{\left( 0\right) },\pi _{n}^{\left(
1\right) },\pi _{n}^{\left( 2\right) },...\right\} ,\;\pi _{n}^{\left(
i\right) }\geq 0,\;\sum_{i=0}^{\infty }\pi _{n}^{\left( i\right) }=1,\;
\end{equation*}%
which are defined on a common probability space $(\Omega ,\mathcal{A},\fP)$.
Moreover, for a given environment $\pi =\left\{ \mathbf{\pi }_{n}\right\} $,
the branching process $\left\{ Z_{n},\;n\geq 0\right\} $ satisfies
\begin{equation*}
Z_{0}=1,\;\fE_{\pi }\left[ s^{Z_{n+1}}|\;Z_{0},Z_{1},...,Z_{n}\right]
=\left( f_{n}\left( s\right) \right) ^{Z_{n}},
\end{equation*}%
where $
f_{n}\left( s\right) =\sum_{i=0}^{\infty }\pi _{n}^{\left( i\right) }s^{i}$
is the generating function of $\pi _{n}$; in other words, $\pi _{n}$ is the
(common) offspring distribution for the particles at generation $n$. Here
and below we use the subscript~$\pi $ to indicate that the expectation (or
probability with the notation~$\fP_{\pi }$) is taken under the given
environment $\pi$. As is shown in various articles on the branching
processes in random environment, the asymptotic behavior of $\{Z_{n}\}$ is
crucially affected by the so-called \textit{associated random walk} $%
\{S_{n}\}$ defined as follows:
\begin{equation*}
S_{0}=0,\;S_{n}:=X_{1}+X_{2}+\cdot \cdot \cdot +X_{n},\;n\geq 1,
\end{equation*}%
where
\begin{equation*}
X_{n}=\log f_{n-1}^{\prime }\left( 1\right) ,n=1,2,\ldots .
\end{equation*}%
are the logarithmic mean offspring numbers. For notational ease let
\begin{equation*}
f\left( s\right) =f_{0}\left( s\right) ,\quad \mbox{and}\quad X=X_{1}.
\end{equation*}%
We call $\{Z_{n}\}$ a \textit{subcritical} branching process in random
environment if
\begin{equation}
\mathbf{E}\left[ X\right] =:-a<0.  \label{Expect}
\end{equation}%
By the SLLN this implies that the associated random walk $\{S_{n}\}$
diverges to $-\infty $ almost surely, which, in view of the inequality $%
\mathbf{P}_{\pi }\left( Z_{n}>0\right) \leq \mathbf{E}_{\pi }\left[ Z_{n}%
\right] =e^{S_{n}},$ leads to almost sure extinction of $\{Z_{n}\}$ .

Subcritical branching processes in \hbox{i.i.d.} random environment have
been considered in a number of articles, see, for instance, \citet{Afanasyev80},
\citet{De88}, \cite{Liu96}, \cite{DH}, \cite{Afanasyev98}, \cite{FV99},
\cite{Afanasyev01}, \cite{GuLi01}, \cite{GKV03}, \cite{Vatutin03}, \cite%
{AGKV05}, \cite{Bansaye08}, \cite{Bansaye09}, \cite{ABKV10}, and \cite{ABKV11}.
According to these papers, a subcritical branching process in random
environment is called \emph{weakly} subcritical if there exists $\theta \in
(0,1)$ such that $\mathbf{E}\left[ Xe^{\theta X}\right] =~0;$ \emph{%
intermediately} subcritical if $\mathbf{E}\left[ Xe^{X}\right] =0;$ and
\emph{strongly} subcritical if $\mathbf{E}\left[ Xe^{X}\right] <0.$

The classification above is not exhaustive, though. An important exception is
that the random variable~$X$ is such that $\mathbf{E}\left[ Xe^{\theta X}%
\right] =\infty $ for any $\theta >0$, and this is the focus of the present
paper. To be more specific, we suppose that $\sigma ^{2}:=\var(X)<\infty $
and, in addition, following the custom of writing $f\sim g$ to mean
that the ratio $f/g$ converges to 1, we have, as $x\rightarrow \infty $,
\begin{equation}
A(x):=\mathbf{P}\left( X>x\right) \sim \frac{l(x)}{x^{\beta }}, \q
\mbox{for some  }\beta
>2,  \label{Ctail1}
\end{equation}%
where $l(x)$ is a function slowly varying at infinity. Thus, the random
variable $X$ does not satisfy the Cramer condition.

We will see below that for the non-Cramer case, similarly to  other cases,
the asymptotics of the survival probability and the growth of the population
size given survival are specified  in main by the behavior of the associated
random walk. However, the influence of the associated random walk for the
non-Cramer case has essentially different nature: for the Cramer cases, the
survival for a long time happens due to the ``atypical'' behavior of the \emph{whole} trajectory of the associated random walk that results in its smaller,
then usually, slope for the strongly subcritical case (\cite{AGKV05}), in its
convergence to a Levy process attaining its minimal value at the end of the
observation interval for the intermediately subcritical case (\cite{ABKV11}),
and in the positivity of its essential part for the weakly subcritical case
(\cite{ABKV10}). For the non-Cramer case, the process survives for a long time
owing to a \emph{single} big jump of the associated random walk at
the beginning of the evolution which, in turn, is accompanied by an explosion
of the population size at this moment; see Lemmas  \ref{lemma:U_n} and \ref{lemma:Z_U_n} for the
precise information. Besides, the number of particles at a
distant moment $n$ given its survival up to this moment tends to infinity
for the non-Cramer case, while for the other types of subcritical processes
in random environment such conditioning leads to discrete limiting
distributions with no atoms at infinity.

One of our assumptions is the following (technical) 
condition for $A(x)$: for any fixed $h>0$,
\begin{equation}
A(x+h) - A(x)
=-\frac{h\beta A(x)}{x}(1+o(1))\mbox{   as   }x\rightarrow \infty.
\label{remainder}
\end{equation}
Next, for any offspring distribution $\widetilde{\pi}=\{\widetilde{\pi}
^{(i)}: i \geq 0\}$ with generating function $\widetilde{f}(s)$, denote
\begin{equation}  \label{dfn:eta}
\eta (\widetilde{\pi})=\frac{\sum_{i=0}^{\infty} i(i-1)\widetilde{\pi}^{(i)}%
} {2(\sum_{i=0}^{\infty }i \widetilde{\pi}^{(i)})^{2}}=\frac{\widetilde{f}%
^{\prime \prime }(1)}{2\left( \widetilde{f}^{\prime}(1)\right) ^{2}}.
\end{equation}
Introduce the following
\begin{asmptn}
\label{asmptn:eta}
\begin{enumerate}[(i)]
\item There exists $\delta >0$ such that, as $x\rightarrow \infty $,
\begin{equation*}
\mathbf{P}\left( \eta(\pi_0) >x\right) =o\left( \frac{1}{\log x\times \left(
\log \log x\right) ^{1+\delta }}\right) .
\end{equation*}

\item
As $x\rightarrow \infty $, (under probability $\mathbf{P}$,)
\begin{equation}
\mathbf{\mathcal{L}}\left( f\left( 1-e^{-x}\right) |X>x\right) \Longrightarrow \mathcal{L}(\gamma),  \label{G1}
\end{equation}
where $\gamma $ is a random variable which is less than $1$ with a positive
probability.
\end{enumerate}
\end{asmptn}

It can be shown that if  $\pi_0$ is either almost surely a Poisson
distribution or almost surely a geometric distribution, and if \eqref{Ctail1} is satisfied, then \eqref{G1} holds with
$\gamma \equiv 0$. Moreover,
it is not difficult to give an example of branching processes in random
environment where $\gamma $ is either positive and less than 1 with
probability 1, or random with support not concentrated at~1. Indeed, let $%
\gamma $ be a random variable taking values in $[0,1-\delta ]\subset \left[
0,1\right] $ for some $\delta \in (0,1]$, and $p$ and~$q,p+q=1,pq>0,$ be
random variables independent of $\gamma $ such that the random variable $
X:=\log \left( 1-\gamma \right) +\log (p/q) $
meets conditions (\ref{Expect}) and~(\ref{Ctail1}). Define $
f(s):=\gamma +\left( 1-\gamma \right) q/(1-ps).$
Then
$f^{\prime }(1)=\left( 1-\gamma \right) p/q=\exp({X}), $
and it is straightforward to show that %
for any $\varepsilon \in \left( 0,1\right) $,
\begin{equation*}
\lim_{x\rightarrow \infty }\mathbf{P}\left( |f(1-e^{-x})-\gamma | \geq \varepsilon \,|\, X>x\right)
\leq \lim_{x\rightarrow \infty }\mathbf{P}\left( X-x\leq -\log \varepsilon \,|\, X>x\right) =0,
\end{equation*}%
therefore \eqref{G1} holds.

Let us briefly explain Assumption 1(ii). For any fixed environment $\pi$ and
any $x>0$, let $\mathcal{L}_{\pi }\left( Z_{1}e^{-x}\right) $ be the
distribution of $Z_{1}e^{-x}$. Note that
this actually only depends on $\pi_0$. We will show in Lemma \ref%
{lemma:char_gamma} that if condition~(\ref{Ctail1}) holds, then (\ref{G1})
is equivalent to the following assumption, concerning weak convergence of \emph{random} measures: (under probability $\mathbf{P}$,)
\begin{equation*}
\mbox{ conditional on $\{X>x\}$, } \mathcal{L}_{\pi }\left(
Z_{1}e^{-x}\right) \Longrightarrow \gamma \delta _{0}+\left( 1-\gamma
\right) \delta _{\infty } \quad\mbox{as } x\rightarrow \infty, 
\end{equation*}
where $\delta _{0}$ and $\delta _{\infty }$ are measures assigning unit
masses to the corresponding points.

In what follows we assume that \textit{the distribution of} $X$ \textit{is
nonlattice}. The case when the distribution of $X$ is lattice needs natural
changes related to the local limit theorem that we use in our proofs (see
Proposition~\ref{Tlocal} below).

Define
\begin{equation}
f_{k,n}\left( s\right) :=f_{k}(f_{k+1}(...(f_{n-1}\left( s\right)
)...)),\;0\leq k\leq n-1,\quad\mbox{and}\quad f_{n,n}\left( s\right) :=s.
\label{DefFF}
\end{equation}
When $k=0$, $f_{0,n}(s)=E_{\pi}(s^{Z_n})$ is the conditional probability generating function of $Z_n$.

The following is our first main theorem which deals with the the survival
probability of the process.

\begin{theorem}
\label{Textinction}Assume conditions \eqref{Expect}, \eqref{Ctail1}, %
\eqref{remainder} and Assumption $\ref{asmptn:eta}$. Then the survival
probability of the process $\{Z_n\}$ has, as $n\rightarrow \infty$, the
asymptotic representation
\begin{equation*}
\mathbf{P}\left( Z_{n}>0\right) \sim K\mathbf{P}\left( X>na\right),
\end{equation*}
where%
\begin{equation}
K:=\sum_{j=0}^{\infty }\mathbf{E}\left[ 1-f_{0,j}\left( \gamma \right) %
\right] \in (0,\infty),  \label{DefK}
\end{equation}%
and $\gamma $ is a random variable that has the same distribution as the $%
\gamma $ in Assumption $\ref{asmptn:eta} ($ii$)$ and is independent of the
underlying environment $\{\pi _{n}\}$ $($and consequently of $\{f_{0,j}\}$ $%
) $.
\end{theorem}

As we mentioned earlier, if $\pi_0$ is either almost surely a Poisson
distribution or almost surely a geometric distribution, then \eqref{Ctail1} implies \eqref{G1} with $\gamma \equiv
0 $, so the constant $K$ becomes $\sum_{j=0}^{\infty }\mathbf{P}(Z_{j}>0)$.
In this case we can give the following intuitive explanation of Theorem \ref{Textinction}:
Let
\begin{equation}
U_{n}=\inf \left\{ j:X_{j}>na\right\}  \label{dfn:U_n}
\end{equation}%
be the first time when the increment of the random walk $S:=\{S_{j},j\geq
0\} $ exceeds~$na$. Then one can show that the event $\{Z_{n}>0\}$ is
asymptotically equivalent to $\{U_{n}<n,Z_{U_{n}-1}>~0\}=\cup
_{j<n}\{Z_{j-1}>0,U_{n}=j\}$. Now for each fixed $j\ge 1$, $\mathbf{P}%
(Z_{j-1}>0,U_{n}=j)$ $\sim$ $\mathbf{P}(Z_{j-1}>0)\mathbf{P}(X>na)$, and
hence, not rigorously,
\begin{equation*}
\mathbf{P}(Z_{n}>0)\sim \mathbf{P}(U_{n}<n,Z_{U_{n}-1}>0)\sim
\sum_{j=1}^{\infty }\mathbf{P}(Z_{j-1}>0)\cdot \mathbf{P}(X>na)=K\mathbf{P}%
(X>na).
\end{equation*}%
In fact, we may say more: (even in the general case when $\gamma \not\equiv
0 ,$) the process survives owing to one big jump of the associated random
walk which happens at the very beginning of the evolution of the process;
moreover, the big jump is accompanied by a population explosion which leads
to survival. See Lemmas \ref{lemma:U_n} and \ref{lemma:Z_U_n} for the
precise information.

The next result gives a Yaglom type conditional limit theorem for the number
of particles up to moment $n$ given survival of the process to this moment.
Recall that $\sigma^2 = \var(X)$, and $U_n$ is defined in \eqref{dfn:U_n}.

\begin{theorem}
\label{thm:FLT}Assume conditions \eqref{Expect}, \eqref{Ctail1}, %
\eqref{remainder} and Assumption~$\ref{asmptn:eta}$. Then for any $j\geq 1$,
\begin{equation*}
\lim_{n\rightarrow \infty }\mathbf{P}(U_{n}=j\ |\ Z_{n}>0)=
\mathbf{E}
(1-f_{0,j-1}(\gamma ))/K.
\end{equation*}%
Moreover,
\begin{equation}
\mathcal{L}\left( \left. \frac{Z_{[nt]\vee {U_{n}}}}{Z_{{U_{n}}}\exp
(S_{[nt]\vee {U_{n}}}-S_{{U_{n}}})},\ 0\leq t\leq 1\right\vert
Z_{n}>0\right) \Longrightarrow \left( 1,\ 0\leq t\leq 1\right) ,  \label{TT}
\end{equation}%
\begin{equation*}
\mathcal{L}\left( \left. \frac{1}{\sigma \sqrt{n}}\left( \log \left(
Z_{[nt]\vee {U_{n}}}/Z_{U_{n}}\right) +nta\right),\ 0\leq t\leq 1\right\vert
Z_{n}>0\right) \Longrightarrow (B_{t},\ 0\leq t\leq 1),
\end{equation*}%
and for any $\varepsilon >0$,
\begin{equation*}
\mathcal{L}\left( \left. \frac{1}{\sigma \sqrt{n}}\left( \log \left(
Z_{[nt]}/Z_{[n\varepsilon ]}\right) +n(t-\varepsilon )a\right),\ \varepsilon
\leq t\leq 1\right\vert Z_{n}>0\right) \Longrightarrow (B_{t}-B_{\varepsilon
},\ \varepsilon \leq t\leq 1),
\end{equation*}%
where the symbol $\Longrightarrow $ means weak convergence in the space $%
D[0,1]$ or $D[\varepsilon ,1]$ endowed with Skorokhod topology, and $B_{t}$
is a standard Brownian motion.
\end{theorem}

Therefore after the population explosion at time $U_{n}$, the population
drops exponentially at rate~$a$, with a fluctuation of order $\exp (O(\sqrt{k%
}))$ with $k$ the number of generations elapsed after the explosion.
Moreover, it follows from \eqref{TT} and the continuous mapping theorem that
\begin{equation*}
\mathcal{L}\left( \left. \log \left( Z_{[nt]\vee {U_{n}}}/Z_{U_{n}}\right)
-(S_{[nt]\vee {U_{n}}}-S_{{U_{n}}}),0\leq t\leq 1\right\vert Z_{n}>0\right)
\Longrightarrow (0,0\leq t\leq 1),
\end{equation*}%
and, therefore, after the big jump, at the logarithmic level the
fluctuations of the population are \emph{completely} described by the
fluctuations of the associated random walk.

\section{Some Preliminary Results}

We list some known results for the random walk $S$ and establish some new
ones.

Define
\begin{equation*}
M_{n}=\max_{1\leq k\leq n}S_{k,}\quad L_{n}=\min_{0\leq k\leq n}S_{k},~\tau
_{n}=\min \left\{ 0\leq k\leq n:S_{k}=L_{n}\right\} ,
\end{equation*}%
\begin{equation*}
\tau (x)=\inf \left\{ k>0:\ S_{k}<-x\right\} ,\quad x\geq 0,
\end{equation*}%
and $\tau =\tau (0)=\inf \left\{ k>0:\ S_{k}<0\right\} $. Further, let
\begin{equation*}
D:=\sum_{k=1}^{\infty }\frac{1}{k}\mathbf{P}\left( S_{k}\geq 0\right) ,
\end{equation*}%
which is clearly finite given conditions \eqref{Expect} and \eqref{Ctail1}.

\begin{proposition}
\label{Ttail} $[$\citet[Theorems 8.2.4, page 376]{BB2005}$]$ Under conditions
\eqref{Expect} and \eqref{Ctail1}, as $n\rightarrow \infty $,
\begin{equation*}
\mathbf{P}\left( L_{n}\geq 0\right) =\mathbf{P}\left( \tau >n\right) \sim
e^{D}\mathbf{P}\left( X>an\right).
\end{equation*}
\end{proposition}

Next, let $Y=X+a$. Then $Y$ is a random variable with nonlattice
distribution, and with zero mean and finite variance. Moreover, as $%
x\rightarrow \infty $, the function%
\begin{equation*}
B(x):=\mathbf{P}(Y>x)\ (=\mathbf{P}(X>x-a) = A(x-a))\ \sim \frac{l(x)}{%
x^{\beta }}, \quad \beta >2,
\end{equation*}%
and satisfies a modified version of \eqref{remainder} by replacing $A(x)$
with $B(x)$.

\begin{proposition}
\label{Tlocal} 
$[$\citet[Theorem 4.7.1, page 218]{BB2005}$]$
Assume \eqref{Ctail1} and \eqref{remainder}. Then
with $\tilde{S}_{n}:=Y_{1}+...+Y_{n},$ where $Y_{i}\overset{d}{=}Y$ and
independent, we have for any $h>0$, uniformly in
$x\geq N\sqrt{n\log (n+1)}$,
as $N\rightarrow \infty $,
\begin{equation*}
\mathbf{P}\left( \tilde{S}_{n}\in \lbrack x,x+h)\right) =\frac{h\beta nB(x)}{x}(1+o(1)).
\end{equation*}
\end{proposition}
The uniformity of $o(1)$ above is understood as that there exists a function $\delta(N)\downarrow 0$ as $N\to\infty$
such that the term $o(1)$  could be replaced by a function $\delta_h(x,n)$ with $|\delta_h(x,n)|\leq \delta(N)$.

Based on Propositions \ref{Ttail} and \ref{Tlocal} we prove the following

\begin{lemma}
\label{LExponent} Assume conditions \eqref{Expect}, \eqref{Ctail1} and %
\eqref{remainder}. Then, as $n\rightarrow \infty $,
\begin{equation*}
\mathbf{E}\left[ e^{S_{n}};\tau _{n}=n\right] =\mathbf{E}\left[
e^{S_{n}};M_{n}<0\right] \sim \frac{K_{1}}{n}\mathbf{P}\left( X>an\right) ,
\end{equation*}%
where%
\begin{equation}
K_{1}:=\frac{\beta }{a}\exp \left\{ \sum_{n=1}^{\infty }\frac{1}{n}\mathbf{E}%
\left[ e^{S_{n}};S_{n}<0\right] \right\}<\infty .  \label{DefK1}
\end{equation}
\end{lemma}

\begin{proof}
The first equality follows from duality. More specifically, the random walks
$\{S_{k}:k=0,1,\ldots ,n\}$ and $\{S_{k}^{\prime
}:=S_{n}-S_{n-k}:k=0,1,\ldots ,n\}$ have the same law, and the event $\{\tau
_{n}=n\}$ for $\{S_{k}\}$ corresponds to the event $\{M_{n}^{\prime }<0\}$
for $\{S_{k}^{\prime }\}.$

Next we evaluate the quantity
\begin{equation}
\mathbf{E}\left[ e^{S_{n}};S_{n}<0\right] =\mathbf{E}\left[
e^{S_{n}};-\left( \beta +2\right) \log n\leq S_{n}<0\right] +O\left(
n^{-\beta -2}\right) .  \label{Fterm}
\end{equation}%
Clearly, for any $h>0$,
\begin{eqnarray*}
&&\sum_{0\leq k\leq \left( \beta +2\right) h^{-1}\log n}
e^{-\left(k+1\right) h}\cdot\mathbf{P}\left( -(k+1)h+an\leq \tilde{S}%
_{n}\leq -kh+an\right) \\
&\leq &\mathbf{E}\left[ e^{S_{n}};-\left( \beta +2\right) \log n\leq S_{n}<0%
\right] \\
&\leq &\sum_{0\leq k\leq \left( \beta +2\right) h^{-1}\log n}e^{-kh} \cdot%
\mathbf{P}\left( -(k+1)h+an\leq \tilde{S}_{n}\leq -kh+an\right) .
\end{eqnarray*}%
By Proposition \ref{Tlocal}, in the range of $k$ under consideration, as $%
n\rightarrow\infty$,
\begin{eqnarray*}
\mathbf{P}\left( -(k+1)h+an\leq \tilde{S}_{n}\leq -kh+an\right)
&=&\frac{h\beta n}{\left( -(k+1)h+an\right) }B\left( -(k+1)h+an\right) (1+o(1)) \\
&=&\frac{h\beta }{a}A\left( an\right) (1+o(1)),
\end{eqnarray*}%
where $o(1)$ is uniform in $0\leq k\leq \left( \beta +2\right) h^{-1}\log n$%
. Now passing to the limit as $n\rightarrow \infty $ we get
\begin{eqnarray*}
h\sum_{k=0}^{\infty }e^{-(k+1)h} &\leq &\lim \inf_{n\rightarrow \infty }%
\frac{a\mathbf{E}\left[ e^{S_{n}};-\left( \beta +2\right) \log n\leq S_{n}<0%
\right] }{\beta A\left( an\right) } \\
&\leq &\lim \sup_{n\rightarrow \infty }\frac{a\mathbf{E}\left[
e^{S_{n}};-\left( \beta +2\right) \log n\leq S_{n}<0\right] }{\beta A\left(
an\right) } \\
&\leq &h\sum_{k=0}^{\infty }e^{-kh}.
\end{eqnarray*}%
Letting now $h\rightarrow 0+$ we see that%
\begin{equation*}
\lim_{n\rightarrow \infty }\frac{a\mathbf{E}\left[ e^{S_{n}};-\left( \beta
+2\right) \log n\leq S_{n}<0\right] }{\beta A\left( an\right) }=1.
\end{equation*}%
Combining this with (\ref{Fterm}) we conclude that, as $n\rightarrow \infty $%
,
\begin{equation}
\mathbf{E}\left[ e^{S_{n}};S_{n}<0\right] =\frac{\beta }{a}A\left( an\right)
(1+o(1))\sim \frac{\beta }{a}\mathbf{P}(X>an).  \label{Nee1}
\end{equation}

Furthermore, we know by a Baxter identity that
\begin{equation*}
\exp\left\{ \sum_{n=1}^{\infty }\frac{t^{n}}{n}\mathbf{E}\left[ e^{S_{n}};S_{n}<0%
\right] \right\}
=1+\sum_{n=1}^{\infty }t^{n}\mathbf{E}\left[ e^{S_{n}};M_{n}<0\right],
\end{equation*}%
see for example Chapter XVIII.3 in \cite{FellerII} or Chapter 8.9 in \cite%
{BGT}. From (\ref{Nee1}) and Theorem 1 in \cite{CNW} we get%
\begin{equation*}
\mathbf{E}\left[ e^{S_{n}};M_{n}<0\right] \sim \frac{K_{1}}{n}\mathbf{P}%
(X>an),
\end{equation*}%
where $K_{1}$ is given by (\ref{DefK1}). That $K_1<\infty$ follows from %
\eqref{Nee1}.
\end{proof}

\begin{corollary}
\label{C_minfinite} Under the conditions of Lemma $\ref{LExponent}$, the
constant $K$ in \eqref{DefK} is finite.
\end{corollary}

\begin{proof}
Clearly,%
\begin{equation}
1-f_{0,j}(\gamma) \leq 1-f_{0,j}(0) =\mathbf{P}_{\pi}(Z_{j}>0)= \min_{0\leq
i\leq j}\mathbf{P}_{\pi }(Z_{i}>0) \leq \min_{0\leq i\leq j}e^{S_{i}} =
e^{S_{\tau _{j}}}.  \label{TT11}
\end{equation}
Thus
\begin{eqnarray*}
K &=&\sum_{j=0}^{\infty }\mathbf{E}\left[ 1-f_{0,j}\left( \gamma \right) %
\right] \leq \sum_{j=0}^{\infty }\mathbf{E}\left[ e^{S_{\tau _{j}}}\right] \\
&=&\sum_{j=0}^\infty \sum_{i=0}^j \mathbf{E}\left[ e^{S_{i}}; \tau _{j} = i%
\right] =\sum_{i=0}^\infty \mathbf{E}\left[ e^{S_{i}}; \tau_{i} = i\right]%
\cdot \sum_{j=i}^\infty \mathbf{P} (L_{j-i}\ge 0).
\end{eqnarray*}
The last term is finite by Lemma \ref{LExponent} and Proposition \ref{Ttail}.
\end{proof}

Next, recall that $U_{n}=\inf \left\{ j:X_{j}>na\right\} $, and $\tau =
\inf\{j>0: S_j <0\}.$ The next result says that if the associated random
walk remains nonnegative for a long time, then there must be a big jump at
the beginning.

\begin{proposition}
\label{Tdurr} $[$\citet[Theorem 3.2, page 283]{Dur1}$]$ If conditions %
\eqref{Expect} and \eqref{Ctail1} hold then%
\begin{equation*}
\lim_{n\rightarrow \infty }\mathbf{P}\left( U_{n}=j|\tau >n\right) =\frac{1}{%
\mathbf{E}\tau }\mathbf{P}\left( \tau >j-1\right) .
\end{equation*}
\end{proposition}

\section{Proof of Theorem\textbf{\ \protect\ref{Textinction}}}

We first introduce the following convergence statements, which will be shown
to be equivalent to each other: (under probability $\fP$,) as $x\rightarrow~\infty,$
\begin{compactenum}[(i)]
\item \label{eq:gamma_x} for any function $\delta (x)$ satisfying
$\lim_{x\rightarrow\infty}\delta(x)=0$,
$\mathcal{L}\left(f\left(1-\exp(-x(1+\delta (x)))\right)
|X>x\right)$ $ \Longrightarrow$ $\mathcal{L}(\gamma)$;

\item \label{G22} for any function $\delta (x)$ satisfying $%
\lim_{x\rightarrow\infty}\delta(x)=0$,
$\mathcal{L}\left(f\left(\exp(-\exp(-x(1+\delta (x))))\right)
|X>x\right)$ $ \Longrightarrow$ $\mathcal{L}(\gamma)$;

\item \label{G3} for any $\lambda >0$, $\mathcal{L}\left( f\left(
\exp \left( -\lambda \exp(-x)\right) \right)
|X>x\right)$ $\Longrightarrow$ $\mathcal{L}(\gamma)$;
\quad and

\item \label{G2} conditional on $\{X>x\}$, $\mathcal{L}_{\pi }\left( Z_{1}e^{-x}\right)
\Longrightarrow \gamma \delta _{0}+\left( 1-\gamma \right) \delta _{\infty}$.
\end{compactenum}

\begin{lemma}
\label{lemma:char_gamma} Assume condition \eqref{Ctail1}. Then the
convergences \eqref{eq:gamma_x} $\sim$\eqref{G2} above are equivalent, and
are all equivalent to~\eqref{G1}.
\end{lemma}

\begin{proof}
We will show that \eqref{G1}$\Rightarrow$\eqref{eq:gamma_x}$\Rightarrow$%
\eqref{G22}$\Rightarrow$\eqref{G3}$\Leftrightarrow$ \eqref{G2}, and finally %
\eqref{G3}$\Rightarrow$\eqref{G1}.

We first prove that \eqref{G1} implies \eqref{eq:gamma_x}.
By \eqref{Ctail1}, the events $\{X>x\}$ and $\{X>x(1+\delta (x))\}$ are
asymptotically equivalent to each other (in the sense that $%
\lim_{x\rightarrow \infty }\mathbf{P}(X>x|X>x(1+\delta
(x))=\lim_{x\rightarrow \infty }\mathbf{P}(X>x(1+\delta (x))|X>x)=1$), hence
it follows, for example, from Lemma 17 in \cite{lz08}, that
\begin{eqnarray*}
\lim_{x\rightarrow \infty }\mathbf{P}\left( f(1-e^{-x(1+\delta
(x))})>y|X>x\right) &=&\lim_{x\rightarrow \infty }\mathbf{P}\left(
f(1-e^{-x(1+\delta (x))})>y|X>x(1+\delta (x))\right) \\
&=&\lim_{x\rightarrow \infty }\mathbf{P}\left( f\left( 1-e^{-x}\right)
>y|X>x\right) .
\end{eqnarray*}

Next we prove that \eqref{eq:gamma_x} implies \eqref{G22}. It is easy to see
that for any $\delta (x)\rightarrow 0$, for all sufficiently large~$x$,
\begin{equation*}
1-\exp (-x(1+\delta (x)+1/x))\geq \exp (-\exp (-x(1+\delta (x))))\geq 1-\exp
(-x(1+\delta (x))),
\end{equation*}%
and consequently, by the monotonicity of $f$,
\begin{equation*}
f(1-\exp (-x(1+\delta (x)+1/x)))\geq f(\exp (-\exp (-x(1+\delta (x)))))\geq
f(1-\exp (-x(1+\delta (x)))).
\end{equation*}%
Now, by \eqref{eq:gamma_x}, taking $\delta (x)$ to be $\delta (x)+1/x$ and $%
\delta (x)$ respectively, we have that both the first and the third random
variables, conditional on $\{X>x\}$, converge in law to $\gamma $. It
follows that the middle random variable also converges, implying \eqref{G22}.

To show \eqref{G3} from \eqref{G22} we simply take $\delta(x)=-\log(%
\lambda)/x$. Moreover, \eqref{G2} and \eqref{G3} are equivalent since $%
f\left( \exp \left( -\lambda e^{-x}\right) \right)$ is the Laplace
transform of $Z_{1}e^{-x}.$

Finally we derive \eqref{G1} from \eqref{G3}. In fact, for all $x$
sufficiently large,
\begin{equation*}
\exp (-\exp (-x))\geq 1-\exp (-x)\geq \exp (-e\cdot \exp (-x)).
\end{equation*}%
Hence, again by the monotonicity of $f$,
\begin{equation*}
f(\exp (-\exp (-x)))\geq f(1-\exp (-x))\geq f(\exp (-e\cdot \exp (-x))).
\end{equation*}%
The convergence in \eqref{G1} then follows from \eqref{G3} by taking $%
\lambda $ to be $1$ and $e$.
\end{proof}

\begin{corollary}
\label{cor:exp_gamma} Assume conditions \eqref{Ctail1} and \eqref{G1}. Then
for any function $\delta (x)$ satisfying $\lim_{x\rightarrow \infty }\delta
(x)~=~0$,
\begin{equation}
\mathbf{E}[\gamma ]=\lim_{x\rightarrow \infty }\mathbf{E}\left[
f(1-e^{-x(1+\delta (x))}))|X>x\right] =\lim_{x\rightarrow \infty }\mathbf{E}%
\left[ \left. \mathbf{E}_{\pi }\left[ e^{-\lambda {Z_{1}}/{e^{x}}}\right]
\right\vert X>x\right] .  \label{eq:gamma_f_01}
\end{equation}%
In particular,
\begin{equation}
\mathbf{E}\left[ \gamma \right] =\lim_{x\rightarrow \infty }\mathbf{P}%
(Z_{1}\leq e^{x(1+\delta (x))}|X>x).  \label{eq:offspring_jump_prob}
\end{equation}
\end{corollary}

\begin{proof}
This follows by applying the dominated convergence theorem to the
convergences in \eqref{eq:gamma_x} and~\eqref{G3}.
\end{proof}

\begin{lemma}
\label{lemma:typical_surv_prob} If \eqref{Expect}, \eqref{Ctail1} and
Assumption $\ref{asmptn:eta}$ are valid, then
\begin{equation}  \label{eq:surv_prob}
\lim_{n\rightarrow \infty }\mathbf{P}\left( \exp (-na-2n/\log n) \leq
\mathbf{P}_{\pi }(Z_{n}>0)\leq \exp (-na+n^{2/3})\right) =1.
\end{equation}
In particular, for any sequence $\delta_n$ such that $n(\delta_n - 2/\log
n)\rightarrow\infty$,
\begin{equation}  \label{eq:surv_certain}
\lim_{n\rightarrow \infty }\mathbf{P}\left( Z_{n}>0\ |\ Z_0 \geq
e^{n(a+\delta_n)}\right) =1.
\end{equation}
\end{lemma}

\begin{proof}
The second claim \eqref{eq:surv_certain} follows directly from the first
one, so we shall only prove~\eqref{eq:surv_prob}.
We have (see, for instance, \cite{GK00}) that
\begin{equation}\label{surv_prob}
\mathbf{P}_{\pi }(Z_{n}>0)=\left(e^{-S_{n}}+\sum_{k=0}^{n-1}g_{k}(f_{k+1,n}(0))e^{-S_{k}}\right) ^{-1},
\end{equation}
where%
\begin{equation*}
g_{k}(s):=\frac{1}{1-f_{k}(s)}-\frac{1}{f_{k}^{\prime }(1)(1-s)}
\end{equation*}%
meets the estimates
\begin{equation*}
0\leq g_{k}(s)\leq 2\eta _{k+1}\quad \mbox{with}\quad \eta _{k+1}:=\eta (\pi
_{k}).
\end{equation*}%

Introduce the events%
\begin{equation}  \label{dfn:Gn_Hn}
G_{n}:= \left\{ \max_{1\leq k\leq n}\left\vert S_{k}+ka\right\vert
<n^{2/3}\right\} \quad \mbox{and} \quad H_{n}:= \left\{ 1+\sum_{k=1}^{n}\eta
_{k}\leq 2ne^{n/\log n}\right\}.
\end{equation}%
By the functional central limit theorem, $\lim_{n\rightarrow\infty}\mathbf{P}%
\left( {G}_{n}\right)=1$. Further, by Assumption~1(i),
\begin{equation*}
1-\mathbf{P}\left( {H}_{n}\right) \leq n\mathbf{P}\left( \eta _{1}\geq
e^{n/\log n}\right) =n\times o\left( \frac{\log n}{n\left( \log \left(
n/\log n\right) \right) ^{1+\delta }}\right) =o\left( \frac{1}{\log ^{\delta
}n}\right) .
\end{equation*}%
Thus, $\lim_{n\rightarrow\infty} \mathbf{P}\left( {H}_{n}\right) = 1$, and,
consequently,
\begin{equation}
\lim_{n\rightarrow\infty}\mathbf{P}\left( G_{n}\cap {H}_{n}\right) =1.
\label{eq:GnHn_occur}
\end{equation}%
It then suffices to show that on the event ${G}_{n}\cap H_{n}$,
\begin{equation*}
\exp (-na-2n/\log n)\leq \mathbf{P}_{\pi }(Z_{n}>0)\leq \exp (-na+n^{2/3})
\end{equation*}%
for all sufficiently large $n$.

The upper bound follows from the evident estimates
\begin{equation*}
\mathbf{P}_{\pi }(Z_{n}>0)\leq \mathbf{E}_{\pi }\left[ Z_{n}\right] =\exp
(S_{n})\leq \exp (-na+n^{2/3}).
\end{equation*}%
As to the lower bound, since $g_{k}(s)\leq 2\eta _{k+1},$ we have
\begin{equation*}
\mathbf{P}_{\pi }(Z_{n}>0)\geq \frac{1}{e^{-S_{n}}+2\sum_{k=0}^{n-1}\eta
_{k+1}e^{-S_{k}}}.
\end{equation*}%
Finally observe that on the event $G_{n}\cap H_{n}$ we have
\begin{eqnarray}
e^{-S_{n}}+2\sum_{k=0}^{n-1}\eta _{k+1}e^{-S_{k}} &\leq
&2e^{na+n^{2/3}}\left( 1+\sum_{k=0}^{n-1}\eta _{k+1}\right)  \notag \\
&\leq &4e^{na+n^{2/3}}\cdot ne^{n/\log n} \leq e^{ na+2n/\log n}
\label{eq:bdd_on_GnHn}
\end{eqnarray}%
for all sufficiently large $n.$
\end{proof}

\begin{lemma}
\label{lemma:jump_surv_prob} Assume \eqref{Expect}, \eqref{Ctail1} and
Assumption $\ref{asmptn:eta}$. Then for any $k\in\mathbb{N}$,
\begin{equation*}
\lim_{n\rightarrow\infty}\mathbf{P}\left( Z_{n}>0\,|\,X_{1}>na, Z_0 = k
\right) = \mathbf{E}\left[ 1-\gamma^k \right].
\end{equation*}
\end{lemma}

\begin{proof}
We only prove for the case when $k=1$. We have
\begin{eqnarray*}
\mathbf{P}\left( Z_{n}>0\,|\,X_{1}>na\right) =\mathbf{E}\left[ \mathbf{P}%
_{\pi }\left( Z_{n}>0\right) \,|\,X_{1}>na\right] =\mathbf{E}\left[%
1-f_{0}(f_{1,n}(0))\,|\,X_{1}>na\right] .
\end{eqnarray*}%
Write
\begin{equation*}
1-f_{1,n}(0)=\mathbf{P}_{\pi }(Z_{n}>0|Z_1=1)=e^{-an(1+\zeta (n))}.
\end{equation*}%
According to the previous lemma, there exists a deterministic function $%
\delta(n)\rightarrow 0$ as $n\rightarrow \infty $ such that
\begin{equation*}
\left\vert \zeta (n)\right\vert \leq \delta(n)
\end{equation*}%
with probability approaching 1 as $n\rightarrow \infty $. The conclusion
then follows from Corollary~\ref{cor:exp_gamma}.
\end{proof}

\begin{lemma}
\label{LRrem}Assume conditions \eqref{Expect}, \eqref{Ctail1}
and \eqref{remainder}. Then for any $\varepsilon >0$ there exists $M$ such
that for all
$n$ sufficiently large,
\begin{equation*}
\mathbf{P}\left( Z_{n}>0;\tau _{n}>M\right) =\mathbf{E}\left[ \mathbf{P}%
_{\pi }(Z_{n}>0);\tau _{n}>M\right] \leq \varepsilon \mathbf{P}\left(
X>na\right) .
\end{equation*}
\end{lemma}

\begin{proof}
Using the inequality $\mathbf{P}_{\pi }(Z_{n}>0)\leq e^{S_{\tau _{n}}} $
established in~\eqref{TT11}, we have by Proposition \ref{Ttail} and Lemma~%
\ref{LExponent} that
\begin{eqnarray*}
&&\mathbf{E}\left[ \mathbf{P}_{\pi }(Z_{n}>0);\tau _{n}>M\right] \leq
\sum_{k=M}^{n}\mathbf{E}\left[ e^{S_{\tau _{n}}};\tau _{n}=k\right] \\
&=&\sum_{M\leq k\leq n/2}\mathbf{E}\left[ e^{S_{k}};\tau _{k}=k\right]
\mathbf{P}\left( L_{n-k}\geq 0\right) +\sum_{n/2<k\leq n}\mathbf{E}\left[
e^{S_{k}};\tau _{k}=k\right] \mathbf{P}\left( L_{n-k}\geq 0\right) \\
&\leq &\mathbf{P}\left( L_{[(n+1)/2]}\geq 0\right) \sum_{k=M}^{\infty }%
\mathbf{E}\left[ e^{S_{k}};\tau _{k}=k\right] +C\frac{\mathbf{P}\left(
X>an\right) }{n}\sum_{k\leq n/2}\mathbf{P}\left( L_{k}\geq 0\right) \\
&\leq &\varepsilon \mathbf{P}\left( X>an\right) .
\end{eqnarray*}
\end{proof}

Now we are ready to prove Theorem \ref{Textinction}.

\begin{proof}[Proof of Theorem \protect\ref{Textinction}]
For fixed $M$ and $N$ we write%
\begin{eqnarray*}
\mathbf{P}\left( Z_{n}>0\right) &=&\mathbf{P}\left( Z_{n}>0;U_{n}\leq
NM\right) \\
&&+\mathbf{P}\left( Z_{n}>0;U_{n}>NM;\tau _{n}>M\right) +\mathbf{P}\left(
Z_{n}>0;U_{n}>NM;\tau _{n}\leq M\right) .
\end{eqnarray*}%
By Lemma \ref{LRrem}, for any $\varepsilon >0$ there exists $M$ such that
for all sufficiently large $n$,
\begin{equation*}
\mathbf{P}\left( Z_{n}>0;U_{n}>NM;\tau _{n}>M\right) \leq \mathbf{P}\left(
Z_{n}>0;\tau _{n}>M\right) \leq \varepsilon \mathbf{P}\left( X>na\right).
\end{equation*}%
Moreover, by Propositions \ref{Tdurr} and \ref{Ttail} there exists $N$ such
that for all sufficiently large $n$,
\begin{eqnarray*}
\mathbf{P}\left( Z_{n}>0;U_{n}>NM;\tau _{n}\leq M\right) &\leq &\mathbf{P}%
\left( U_{n}>NM;\tau _{n}\leq M\right) \\
&=&\sum_{k=0}^{M}\mathbf{P}\left( U_{n}>NM;\tau _{n}=k\right) \\
&\leq &\sum_{k=0}^{M}\mathbf{P}\left( \tau _{k}=k\right) \mathbf{P}\left(
U_{n}>\left( N-1\right) M;\tau >n-M\right) \\
&\leq &(M+1)\mathbf{P}\left( U_{n}>\left( N-1\right) M;\tau >n-M\right) \\
&\leq &\varepsilon \mathbf{P}\left( \tau >n-M\right)
\leq 2e^D\varepsilon \mathbf{P}\left( X>na\right) .
\end{eqnarray*}
Hence the main contribution to $\mathbf{P}\left( Z_{n}>0\right)$ comes from $%
\mathbf{P}\left( Z_{n}>0;U_{n}\leq NM\right)$, i.e., when there is a big
jump of the associated random walk at the beginning.

To proceed, we introduce the events
\begin{equation}
A_{k}=A_{k}(n):=\left\{ X_{i}\leq na,1\leq i\leq k\right\}, \q k=1,2,\ldots,n.
\label{A_events}
\end{equation}%
For each fixed $j$, we have by the Markov property
\begin{eqnarray}
&&\mathbf{P}( Z_{n}>0;U_{n}=j)\notag\\
 &=&\mathbf{P}(A_{j-1};X_{j}>na;Z_{n}>0)  \notag
\\
&=&\sum_{k=1}^{\infty }\mathbf{P}(A_{j-1}; Z_{j-1}=k) \mathbf{P}(X_1>na)%
\mathbf{P}\left( Z_{n-j+1}>0|X_1>na;Z_{0}=k\right).  \label{eq:surv_Un}
\end{eqnarray}%
Clearly, $\mathbf{P}(A_{j-1},Z_{j-1}=k)=\mathbf{P}(A_{j-1}(n),Z_{j-1}=k)$
increases to $\mathbf{P}(Z_{j-1}=k)$ as $n\rightarrow \infty $. Dividing
both sides of \eqref{eq:surv_Un} by $\mathbf{P}\left( X_{1}>na\right) $ and
applying the dominated convergence theorem and Lemma \ref{lemma:jump_surv_prob} yield
\begin{eqnarray}
\frac{\mathbf{P}\left( Z_{n}>0;U_{n}=j\right) }{\mathbf{P}\left(
X_{1}>na\right) } &=&\sum_{k=1}^{\infty }\mathbf{P}\left(
Z_{j-1}=k;A_{j-1}\right) \mathbf{P}\left( Z_{n-j+1}>0|X_{1}>na,Z_{0}=k\right)
\notag \\
&\sim &\sum_{k=1}^{\infty }\mathbf{P}\left( Z_{j-1}=k\right) \mathbf{E}\left[
1-\gamma ^{k}\right]  \notag \\
&=&\mathbf{E}\left[ 1-f_{0,j-1}\left( \gamma \right) \right] .  \label{RRR}
\end{eqnarray}%
The last equality holds because of the independence assumption that we put
on $\gamma $ and $\{f_{0,j}\}$. Hence%
\begin{equation}
\mathbf{P}\left( Z_{n}>0\right) \sim \sum_{j=1}^{\infty }\mathbf{E}\left[
1-f_{0,j-1}\left( \gamma \right) \right] \cdot \mathbf{P}\left( X>na\right)
\label{RRR1}
\end{equation}%
as desired.
\end{proof}

\section{Functional Limit Theorems Conditional on Non-extinction}

The proof of Theorem \ref{Textinction} shows that the process survives owing
to one big jump of the associated random walk which happens at the very
beginning of the evolution of the process. This motivates the study of the
conditional distribution of $U_{n}$ given $\{Z_{n}>0\}$, which is the
content of the next lemma.

\begin{lemma}
\label{lemma:U_n} For any $j\geq 1$,
\begin{equation}
\lim_{n\rightarrow \infty }\mathbf{P}(U_{n}=j|Z_{n}>0)
=\mathbf{E}\left[1-f_{0,j-1}(\gamma )\right] /K,
\label{eq:U_n_limit}
\end{equation}%
where $K$ is as in Theorem $\ref{Textinction}$. In particular, for any $%
\varepsilon>0$, there exists $M>0$ such that
\begin{equation}
\limsup_{n\rightarrow \infty }\mathbf{P}(U_{n}>M|Z_{n}>0)\leq \varepsilon .
\label{eq:U_n_UAC}
\end{equation}
\end{lemma}

\begin{proof}
The first claim follows from the representation
\begin{equation*}
\mathbf{P}(U_{n}=j|Z_{n}>0)=\frac{\mathbf{P}(Z_{n}>0;U_{n}=j)}{\mathbf{P}%
(X_{1}>na)}\times \frac{\mathbf{P}(X_{1}>na)}{\mathbf{P}(Z_{n}>0)}
\end{equation*}%
and relationships (\ref{RRR}) and (\ref{RRR1}).

Estimate (\ref{eq:U_n_UAC}) follows from \eqref{eq:U_n_limit} and Corollary %
\ref{C_minfinite}.
\end{proof}

Thus, we have demonstrated that given survival to time $n$, there must be a
big jump at the early time period. Next lemma complements this by showing
that for survival of the process such a big jump will be accompanied by a
population explosion.

Let $Z_{j}(i)$ be the offspring size of the $i$-th particle existing in
generation $j-1$, and, as we shall deal with $\max_{1\leq i\leq
Z_{U_{n}-1}}Z_{U_{n}}(i)$ repeatedly,  define
\begin{equation}
N_{U_{n}}:=\max_{1\leq i\leq Z_{U_{n}-1}}Z_{U_{n}}(i).
\label{dfn:offspring_U_n}
\end{equation}

\begin{lemma}
\label{lemma:Z_U_n} For any sequence $h_n$ such that $h_n\leq n$ and $h_n\to\infty$,
and $\delta _{n}\rightarrow 0$ satisfying $n(\delta _{n}-2/\log n)\rightarrow \infty $,
\begin{equation*}
\lim_{n\rightarrow \infty }\mathbf{P}\left( U_{n}<h_n, N_{U_{n}}\geq
e^{n(a+\delta_n)}\,|\,Z_{n}>0\right) =1.
\end{equation*}
\end{lemma}

\begin{proof}
We first estimate the probability
\begin{eqnarray}
&&\mathbf{P}\left( U_{n}<h_n, N_{U_{n}}\geq e^{n(a+\delta_n)},
Z_n>0\right)  \notag \\
&=&\sum_{j<h_n}\sum_{k=1}^{\infty }\mathbf{P} (U_{n}=j,Z_{j-1}=k) \cdot
\mathbf{P}\left( N_{U_{n}}\geq e^{n(a+\delta_n)}\,|\,U_{n}=j,Z_{j-1}=k\right)  \notag \\
&& \cdot\ \mathbf{P} (Z_n >0 \,|\, N_{U_{n}}\geq e^{n(a+\delta_n)}, U_{n}=j, Z_{j-1}=k).
\label{eq:surv_decomp}
\end{eqnarray}
Recall the events $A_j(n)$ that we defined in \eqref{A_events}. As $%
n\rightarrow \infty $,
\begin{equation*}
\mathbf{P}(U_{n}=j,Z_{j-1}=k) =\mathbf{P}(Z_{j-1}=k, A_{j-1}(n))\cdot
\mathbf{P}(X>an) \sim \mathbf{P}(Z_{j-1}=k)\cdot \mathbf{P}(X>an).
\end{equation*}
Moreover, by (a simple generalization of) \eqref{eq:offspring_jump_prob},
for any fixed $j$ and $k$,
\begin{equation*}
\lim_{n\rightarrow \infty }\mathbf{P}\left( N_{U_{n}}\geq e^{n(a+\delta_n)}\,|\,U_{n}=j,Z_{j-1}=k\right) =\mathbf{E}[1-\gamma ^{k}].
\end{equation*}
Finally, by \eqref{eq:surv_certain} in Lemma \ref{lemma:typical_surv_prob},
\begin{equation*}
\mathbf{P} (Z_n >0 \,|\, N_{U_{n}}\geq e^{n(a+\delta_n)}, U_{n}=j, Z_{j-1}=k)
\rightarrow 1.
\end{equation*}
Dividing both sides of \eqref{eq:surv_decomp} by $\mathbf{P}(Z_n>0)$ and
applying Theorem \ref{Textinction} and Fatou's lemma we get the conclusion.
\end{proof}

The arguments above lead to the following lemma, which says that
conditioning on $\{Z_{n}>0\}$ is asymptotically equivalent to conditioning
on $\{U_{n}<h_n,N_{U_{n}}\geq e^{n(a+\delta_n)}\}$.

\begin{lemma}
\label{lemma:conditioning}  For any
sequence $h_n$  such that $h_n\leq n$ and $h_n\to\infty$,
and $\delta _{n}\rightarrow 0$
satisfying $n(\delta _{n}-2/\log n)\rightarrow \infty $,
\begin{equation}
||\mathbf{P}(\cdot \,|\,Z_{n}>0)-\mathbf{P}(\cdot
\,|\,U_{n}<h_n,N_{U_{n}}\geq e^{n(a+\delta_n)})||_{TV}\rightarrow 0,
\label{eqn:tv_diff}
\end{equation}%
where $||\cdot ||_{TV}$ denotes the total variation distance.
\end{lemma}

\begin{proof}
This follows from Lemma 17 in \cite{lz08}, 
relation \eqref{eq:surv_certain}, and the previous lemma.
\end{proof}

Lemma \ref{lemma:conditioning} shows that to prove the functional limit
theorems for the population size up to moment~$n$ conditioned on survival of
the process to this moment, we need only to establish limit theorems under
the condition $\{U_{n}<h_n,N_{U_{n}}\geq e^{n(a+\delta_n)}\}$.

\begin{lemma}
\label{lemma:var}
For all sufficiently large $n$, on the event $G_{n}\cap H_{n}$ $($as defined
in \eqref{dfn:Gn_Hn}$)$ we have
\begin{equation*}
\frac{\mathbf{E}_{\pi }\left[ Z_{n}^{2}\right] }{\exp (2S_{n})}\leq \exp
(na+2n/\log n).
\end{equation*}
\end{lemma}

\begin{proof}
Recall the generating functions $f_{k,n}(\cdot)$'s that we defined in %
\eqref{DefFF}. Clearly,
\begin{equation*}
f_{0,n}^{\prime }(s)=\prod_{k=0}^{n-1}f_{k}^{\prime }(f_{k+1,n}(s)),
\end{equation*}%
and%
\begin{equation*}
f_{0,n}^{\prime \prime }(s)
=\prod_{i=0}^{n-1}f_{i}^{\prime}(f_{i+1,n}(s))\cdot \sum_{k=0}^{n-1}\frac{%
f_{k}^{\prime \prime}(f_{k+1,n}(s))}{f_{k}^{\prime}(f_{k+1,n}(s))}
\prod\limits_{j=k+1}^{n-1}f_{j}^{\prime }(f_{j+1,n}(s)).
\end{equation*}%
Hence, letting $s=1$ we get%
\begin{equation*}
\mathbf{E}_{\pi }\left[ Z_{n}(Z_{n}-1)\right] = f_{0,n}^{\prime \prime }(1)
=2e^{2S_{n}}\cdot \sum_{k=0}^{n-1}\eta_{k+1}
e^{-S_{k+1}}.
\end{equation*}%
Thus,
\begin{equation*}
\mathbf{E}_{\pi }\left[ Z_{n}^{2}\right] =f_{0,n}^{\prime \prime }(1)+%
\mathbf{E}_{\pi }\left[ Z_{n}\right] =2e^{2S_{n}}\cdot \sum_{k=0}^{n-1}\eta_{k+1}
e^{-S_{k+1}}
+e^{S_{n}}
\end{equation*}%
implying%
\begin{equation*}
\frac{\mathbf{E}_{\pi }\left[ Z_{n}^{2}\right] }{\exp (2S_{n})}=2\cdot
\sum_{k=0}^{n-1}\eta_{k+1}
e^{-S_{k+1}}
+e^{-S_{n}}.
\end{equation*}%
The needed conclusion then follows
from an  argument similar to
\eqref{eq:bdd_on_GnHn}.
\end{proof}

For the following two lemmas we fix a sequence $h_n$ such that
\begin{equation}\label{eq:h}
h_n \to \infty, \q\mbox{and}\q  h_n/n \to 0.
\end{equation}
A simple  example of such a choice is $h_n = \log n.$ We also take $\delta_{n}$ to be a sequence satisfying
\begin{equation} \label{eq:delta}
\delta_{n}\rightarrow 0\q\mbox{and}\q n(\delta_{n}-2/\log n)\rightarrow~\infty.
\end{equation}

\begin{lemma}
\label{lemma:flt}
Suppose \eqref{eq:h} and \eqref{eq:delta} hold.
Then
\begin{equation*}
\aligned
&\mathcal{L}\left( \left. \frac{Z_{[nt]\vee U_{n}}}{Z_{U_{n}}\exp
(S_{[nt]\vee U_{n}}-S_{U_{n}})},0\leq t\leq 1\right\vert
U_{n}<h_n,N_{U_{n}}\geq e^{n(a+\delta_n)}\right)\\
 &\qquad\qquad\qquad\qquad\qquad\qquad\qquad\qquad\qquad\qquad\Longrightarrow
 \left(1,0\leq t\leq 1\right).
\endaligned
\end{equation*}
\end{lemma}

\begin{proof}
Let $\pi ^{\prime }= 
\{\pi _{U_{n}},\pi _{U_{n}+1},\ldots \}$ be the random environment after
time $U_{n}$, and $S_{m}^{\prime}:=S_{m+U_{n}}-S_{U_{n}}$ be the random walk
after time $U_{n}$. Further, define $G_{n}^{\prime }$ and $H_{n}^{\prime }$
for the random environment~$\pi ^{\prime }$ in the same way as in %
\eqref{dfn:Gn_Hn}. Then, by \eqref{eq:GnHn_occur}, as $n\rightarrow \infty $%
, $G_{n}^{\prime }\cap H_{n}^{\prime }$ occurs with probability approaching
one, so we need only to prove the convergence on the event $G_{n}^{\prime
}\cap H_{n}^{\prime }$.

We first prove the marginal convergence, by a mean-variance calculation.
Denote $k=[nt]$.
By our assumption $U_{n}\leq h_n$ for an $h_n/n\to 0$, implying $k=[nt]>U_{n}$ for
all $t>0$ and for all $n$ big enough. Hence
\begin{equation*}
\mathbf{E}_{\pi }(Z_{k}|Z_{{U_{n}}})=Z_{U_{n}}\cdot \exp (S_{k-{U_{n}}%
}^{\prime });
\end{equation*}%
moreover, by Lemma \ref{lemma:var}, for all $n$ big enough, on the event $%
G_{n}^{\prime }\cap H_{n}^{\prime }$,
\begin{equation*}
\var_{\pi }(Z_{k}|Z_{U_{n}}) =Z_{U_{n}}\cdot \var_{\pi
^{\prime}}(Z_{k-U_{n}}|Z_{0}=1)\leq Z_{U_{n}} \cdot \exp (2S_{k-{U_{n}}%
}^{\prime})\cdot \exp (na+2n/\log n).
\end{equation*}%
Consequently, when $N_{U_{n}}\geq \exp(n(a+\delta_n))$ and, therefore, $%
Z_{U_{n}}\geq \exp(n(a+\delta_n))$, we have
\begin{equation*}
\var_{\pi }\left( \left. \frac{Z_{k}}{Z_{U_{n}}\cdot \exp (S_{k-{U_{n}}%
}^{\prime })}\right\vert Z_{U_{n}}\right) \leq \frac{\exp (na+2n/\log n)}{%
Z_{U_{n}}}\leq \frac{\exp (na+2n/\log n)}{\exp(n(a+\delta_n))}%
\rightarrow 0.
\end{equation*}

Next, by Slutsky's theorem (see, e.g., \cite{Ferguson96}) we have
convergence of finite dimensional distributions. Furthermore, since $%
Z_{[nt]\vee {U_{n}}}/(Z_{U_{n}}\exp (S_{[nt]\vee {U_{n}}-U_{n}}^{\prime }))$
are martingales (with respect to the post-$U_{n}$ sigma field $\mathcal{F}%
_{([nt]\vee U_{n})}^{\pi }$, where $\mathcal{F}_{i}^{\pi }=\sigma \langle
\pi ;Z_{j},j\leq i\rangle $), to prove the convergence in the space $D[0,1]$
we need only to show, by Proposition~1.2 in \cite{Aldous89}, the uniform
integrability of $\{Z_{[nt]\vee {U_{n}}}/(Z_{U_{n}}\exp (S_{[nt]\vee {U_{n}}%
-U_{n}}^{\prime }))\}$ for any fixed $t$. This follows from the above
calculation, demonstrating that the elements of the martingale sequence in
question are bounded in $L^{2}$.
\end{proof}

The lemma just proved serves as an LLN type result; the following lemma
gives the CLT type statement.

\begin{lemma}
\label{lemma:flt_2nd}
Suppose \eqref{eq:h} and \eqref{eq:delta} hold. Then
\begin{equation*}
\aligned
&\mathcal{L}\left( \left. \frac{1}{\sigma \sqrt{n}}\left( \log \left(
Z_{[nt]\vee {U_{n}}}/Z_{U_{n}}\right) +nta\right) ,0\leq t\leq 1\right\vert
U_{n}\leq h_n, N_{U_{n}}\geq e^{n(a+\delta_n)}\right)\\
  &\qquad\qquad\qquad\qquad\qquad\qquad\qquad\qquad\qquad\qquad\qquad\quad\Longrightarrow (B_{t},0\leq t\leq 1),
 \endaligned
\end{equation*}%
and for any $\varepsilon >0$,
\begin{equation*}
\aligned
&\mathcal{L}\left( \left. \frac{1}{\sigma \sqrt{n}}\left( \log \left(
Z_{[nt]}/Z_{[n\varepsilon ]}\right) +n(t-\varepsilon )a\right) ,\varepsilon
\leq t\leq 1\right\vert U_{n}\leq h_n, N_{U_{n}}\geq e^{n(a+\delta_n)}
\right)\\
 &\qquad\qquad\qquad\qquad\qquad\qquad\qquad\qquad\qquad\qquad\qquad\Longrightarrow (B_{t}-B_{\varepsilon },\varepsilon \leq t\leq 1),
 \endaligned
\end{equation*}
where $\sigma $ is the standard deviation of $X$, and $B_{t}$ is a standard
Brownian motion.
\end{lemma}

\begin{proof}
By Lemma \ref{lemma:flt} and the continuous mapping theorem,
\begin{equation*}
\aligned &\mathcal{L}\left( \left. \log \left( Z_{[nt]\vee {U_{n}}%
}/Z_{U_{n}}\right) -(S_{[nt]\vee U_{n}}-S_{U_{n}}),0\leq t\leq 1\right\vert
U_{n}\leq h_n,N_{U_{n}}\geq e^{n(a+\delta_n)}\right) \\
&\qquad\qquad\qquad\qquad\qquad\qquad\qquad\qquad\qquad\qquad\qquad\qquad\qquad\Longrightarrow  \left( 0,0\leq t\leq 1\right).
\endaligned
\end{equation*}%
The conclusions then follow from the standard functional central limit
theorem for the post-$U_n$ random walk $\{S_{[nt]\vee U_{n}}-S_{U_{n}}\}$.
\end{proof}

Now we are ready to prove Theorem \ref{thm:FLT}.

\begin{proof}[Proof of Theorem \protect\ref{thm:FLT}]
The first claim follows from Lemma \ref{lemma:U_n}. The second statement
follows from Lemmas \ref{lemma:conditioning}, \ref{lemma:flt} and \ref%
{lemma:flt_2nd}.
\end{proof}

\section*{Acknowledgments}
We thank the anonymous referee for his/her valuable suggestions that help improve the article.

\bibliographystyle{elsarticle-harv}







\end{document}